\documentclass[reqno]{amsart}

\usepackage{amssymb}
\usepackage{amsmath}
\usepackage{amscd}
\usepackage{psfrag}
\usepackage{graphicx}
\usepackage[all]{xy}

\newcommand{\R}{\mathbb{R}}
\newcommand{\Z}{\mathbb{Z}}

\newcommand{\rk}{\operatorname{rank}}

\newcommand{\Int}{\mathop{\mathrm{Int}}\nolimits}
\newcommand{\id}{\mathop{\mathrm{id}}\nolimits}

\renewcommand{\tilde}{\widetilde}

\renewcommand{\leq}{\leqslant}
\renewcommand{\geq}{\geqslant}

\def\spmapright#1{\smash{%
 \mathop{\hbox to 1.3cm{\rightarrowfill}}
  \limits^{#1}}}
\def\spmapleft#1{\smash{%
 \mathop{\hbox to 1.3cm{\leftarrowfill}}
  \limits^{#1}}}

\theoremstyle{plain}
\newtheorem{theorem}{Theorem}[section]
\newtheorem{corollary}[theorem]{Corollary}
\newtheorem{proposition}[theorem]{Proposition}
\newtheorem{lemma}[theorem]{Lemma}

\theoremstyle{definition}
\newtheorem{definition}[theorem]{Definition}

\theoremstyle{remark}
\newtheorem{remark}[theorem]{Remark}

\newtheorem{problem}[theorem]{Problem}

\makeatletter
 
 \@addtoreset{equation}{section}
\makeatother
\title[New examples of Neuwirth--Stallings pairs]{New 
examples of Neuwirth--Stallings pairs and non-trivial
real Milnor fibrations}
\author{R.~Ara\'ujo dos Santos, M.A.B.~Hohlenwerger,
O.~Saeki and T.O.~Souza}
\address{
R.~Ara\'ujo dos Santos, M.A.B.~Hohlenwerger:
Universidade de S\~ao Paulo,
ICMC, Av.\ Trabalhador S\~ao Carlense, 400, 
Centro Postal Box 668, 13560-970, S\~ao Carlos, SP, Brazil}
\email{rnonato@icmc.usp.br, amelia@icmc.usp.br}
\address{
O.~Saeki: Institute of Mathematics for Industry, Kyushu University,
Motooka 744, Nishi-ku, Fukuoka 819-0395,
Japan}
\email{saeki@imi.kyushu-u.ac.jp}
\address{
T.O.~Souza: Faculdade de Matem\'atica, Universidade Federal de 
Uberl\^{a}ndia, Campus Santa M\^onica - Bloco 1F - Sala 1F120,
Av.\ Jo\~ao Naves de Avila, 2121,
Uberl\^andia, MG, CEP: 38.408-100,
Brazil
}
\email{olitaciana@gmail.com}



\date{\today}

\begin{document}

\begin{abstract}
We use topology of configuration spaces to give a characterization 
of Neuwirth--Stallings pairs $(S^5, K)$ with $\dim K = 2$.
As a consequence, we construct
polynomial map germs 
$(\mathbb{R}^{6},0) \to (\mathbb{R}^{3},0)$ 
with an isolated singularity at the origin such
that their Milnor fibers are not diffeomorphic to a disk,
thus putting an end to
Milnor's non-triviality question. 
Furthermore, for a polynomial map germ
$(\mathbb{R}^{2n},0) \to (\mathbb{R}^{n},0)$ or
$(\mathbb{R}^{2n+1},0) \to (\mathbb{R}^{n},0)$, $n \geq 3$,
with an isolated singularity at the origin,
we study the conditions under which
the associated Milnor fiber has the homotopy type
of a bouquet of spheres. We then construct,
for every pair $(n, p)$ with $n/2 \geq p \geq 2$, a
new example of a
polynomial map germ $(\mathbb{R}^n,0) \to (\mathbb{R}^p,0)$ 
with an isolated singularity at the origin such that
its Milnor fiber has the homotopy type of a bouquet of
a positive number of spheres.
\end{abstract}

\maketitle

\section{Introduction}\label{section1}

In the book ``Singular points of complex hypersurfaces'' \cite{Mi},
John Milnor studied the topology of complex
polynomial function germs in terms of the associated locally trivial
fiber bundles. He also showed
the existence of such structures for real
polynomial map germs with an isolated singularity as follows.

\begin{theorem}[{\cite[Theorem~11.2]{Mi}}]\label{TeoMilnor}
Let $f: (\mathbb{R}^n,0) \to (\mathbb{R}^p,0)$,
$n \geq p \geq 2$, be a polynomial map germ
with an isolated singularity at the origin. Then,
there exists an $\varepsilon_0 > 0$ such that 
for all $0 < \varepsilon \leq \varepsilon_{0}$,
the complement of an open tubular neighborhood of the link $K =
f^{-1}(0) \cap S^{n-1}_{\varepsilon}$ in $S^{n-1}_{\varepsilon}$ 
is the total space of a smooth fiber bundle over the sphere $S^{p-1}$, 
with each fiber $F_{f}$ being a smooth compact $(n-p)$-dimensional manifold
bounded by a copy of $K$, where $S^{n-1}_\varepsilon$ denotes
the sphere in $\mathbb{R}^n$ with radius $\varepsilon$ 
centered at the origin.
\end{theorem}

By using the conical structure of the singularity, Milnor proved 
that the diffeomorphism type of the link does not change 
for all $\varepsilon > 0$ small enough. Moreover, since the 
origin is an isolated singularity, we have that 
$0 \in \mathbb{R}^{p}$ is a regular value of
$f|S^{n-1}_{\varepsilon} : S^{n-1}_{\varepsilon} \to \mathbb{R}^{p}$.
Therefore, if the link $K=f^{-1} (0) \cap S^{n-1}_{\varepsilon}$ 
is not empty, then it is a smooth $(n-p-1)$-dimensional
submanifold of the sphere with trivial normal bundle. It also implies 
that, for each fixed $\varepsilon$ one can find a small enough 
$\delta$, $0 < \delta \ll \varepsilon$, and a closed 
disk $D_{\delta}^{p}$ centered at the origin in $\mathbb{R}^{p}$
with radius $\delta$, such that the restriction map 
$f: f^{-1}(D_{\delta}^{p})\cap S_{\varepsilon}^{n-1} 
\to D_{\delta}^{p}$ is a smooth trivial fiber bundle, 
which implies the triviality of the fibration 
$$f: f^{-1}(D_{\delta}^{p} \setminus \{0\})\cap 
S_{\varepsilon}^{n-1} \to D_{\delta}^{p} \setminus \{0\}.$$
By composing this with the radial projection
$\displaystyle{\pi: D_{\delta}^{p} \setminus \{0\} 
\to S^{p-1}_{\delta}}$ onto the boundary sphere
and scaling to the unit sphere, 
one finds that the bundle structure on a neighborhood of 
the link $K$ is given by 
$$\frac{f}{\|f\|}:f^{-1}(D_{\delta}^{p} \setminus \{0\})
\cap S_{\varepsilon}^{n-1} \to S^{p-1}.$$

Now, one can glue this fiber bundle with that given in 
Theorem~\ref{TeoMilnor} along the common boundary 
$f^{-1}(S_{\delta}^{p-1})\cap S_{\varepsilon}^{n-1}$
in a smooth way, so that we get a smooth locally trivial 
fiber bundle 
$$S_{\varepsilon}^{n-1}\setminus K \to S^{p-1}.$$

\begin{remark}
Following Milnor's proof of Theorem~\ref{TeoMilnor},
one sees that, although no precise information about the bundle 
projection above was provided, in the real settings in general 
one cannot expect that it is given by $f/\|f\|$ outside a 
neighborhood of the link. 
See Milnor's example \cite[p.~99]{Mi} or the following 
example adapted from Milnor's one.
%
Consider $f:(\mathbb{R}^{3},0)\to (\mathbb{R}^{2},0)$ given by
$$f(x,y,z)=(x, x^2+yx^2+y^3+yz^2).$$ 
It is easy to see that the singular point set is given by
$\Sigma(f)=\{(0,0,0)\}$, and for all 
$\varepsilon>0$ small enough, we have
$K=\{(0,0,\varepsilon), (0,0,-\varepsilon)\}$.
However, $f/\|f\|$ does have singular points
and therefore it is not a submersion.
\end{remark}

\begin{definition}[Looijenga \cite{Lo}] 
Let $K = K^{n-p-1}$ be an oriented
submanifold of dimension $n-p-1$ of 
the oriented sphere $S^{n-1}$ with trivial normal
bundle, or let $K=\emptyset$. Suppose that for some trivialization
$c: N(K) \to K \times D^p$ of a tubular neighborhood $N(K)$ of $K$,
the fiber bundle defined by the composition
$$N(K) \setminus K \stackrel{c}{\to} 
K \times (D^p \setminus \{0\}) \stackrel{\pi}{\to} S^{p-1},$$
with the last projection being given by $\pi(x,y) = y/\|y\|$,
extends to a smooth fiber bundle $S^{n-1} \setminus K \to S^{p-1}$.
Then, the pair $(S^{n-1}, K^{n-p-1})$ is called a
\emph{Neuwirth--Stallings pair}, or an \emph{NS-pair} for short.
\end{definition}

According to Theorem~\ref{TeoMilnor} and the previous discussion, 
for all $\varepsilon >0 $ sufficiently
small, the pair $(S^{n-1}_{\varepsilon}, f^{-1}(0) 
\cap S^{n-1}_{\varepsilon})$ is
an NS-pair. In this case Looijenga called it the 
\emph{NS-pair associated to the singularity}.

More recently, several generalizations of such a structure have been 
obtained. For instance, in \cite{AT} the authors considered a 
real analytic map germ $f:(\mathbb{R}^{n},0) \to (\mathbb{R}^{p},0)$ 
with non-isolated singularities at the origin and introduced 
a condition which also ensures that the pair 
$(S^{n-1}_{\varepsilon}, f^{-1}(0) \cap S^{n-1}_{\varepsilon})$ 
is an NS-pair with the bundle projection given by
$f/\|f\| : S^{n-1}_{\varepsilon} \setminus K \to S^{p-1}$,
where the link $K = f^{-1}(0) \cap S^{n-1}_{\varepsilon}$
is a smooth manifold.
It was called a \emph{higher open book structure} of the sphere $S_{\varepsilon}^{n-1}$. In \cite{ACT, ACT1} 
it was shown an extension of such structures for spheres of 
small and big enough radii (at infinity), but allowing 
singularity in the ``binding'' $K$. 
In this case, it was called a \emph{singular open book structure} 
of the sphere. 

As pointed out by Milnor in \cite[p.~100]{Mi}, the hypothesis of 
Theorem~\ref{TeoMilnor} is so strong that examples are difficult to 
find, and he posed the following question.

\begin{problem}\label{problem}
For which dimensions $n \geq p \geq 2$ 
do non-trivial examples exist ?
\end{problem}

Milnor did not exactly specify what ``trivial'' means here:
however, he proposed to say that a real polynomial map germ 
$f: (\mathbb{R}^n,0) \to (\mathbb{R}^p,0)$ is \emph{trivial} if 
the fiber $F_{f}$ of the bundle given in Theorem~\ref{TeoMilnor}
is diffeomorphic to a closed disk $D^{n-p}$. 
In particular, this implies that the fibers of the 
associated NS-pair are diffeomorphic to the $(n-p)$-dimensional 
open disk.

\begin{remark}
For a holomorphic function germ $f:(\mathbb{C}^{n+1},0) \to 
(\mathbb{C},0)$ with an isolated singularity
at the origin, it follows from \cite[Appendix B]{Mi} that the 
fibers of the associated Milnor fibration are diffeomorphic to a 
$2n$-dimensional disk if and only if $0$ is a non-singular point 
of $f$; in fact, the function germ $f$ is trivial if and 
only if the Milnor number $\mu_{f} = 
\mathrm{deg}_{0}(\nabla f(z))$ is equal to zero, where 
$\mathrm{deg}_{0}(\nabla f(z))$ stands for the topological 
degree of the map 
$$\varepsilon\frac{\nabla f}{\|\nabla f \|}:
S^{2n+1}_{\varepsilon} \to S^{2n+1}_{\varepsilon}$$
for all $\varepsilon>0$ small enough, and
$$\nabla f = \left(\frac{\partial f}{\partial z_1},
\frac{\partial f}{\partial z_2}, \ldots,
\frac{\partial f}{\partial z_{n+1}}\right).$$
\end{remark}

In \cite{CL} Church and Lamotke used results of 
Looijenga \cite{Lo} and answered the above question in the following
way.

\begin{theorem}\label{T2} {\rm(a)} 
For $0 \leq n-p \leq 2$, non-trivial examples occur precisely
for the dimensions $(n,p)\in \{(2,2), (4,3), (4,2)\}$.

\noindent {\rm(b)} 
For $n-p \geq 4$, non-trivial examples occur for
all $(n,p)$.

\noindent {\rm(c)} For $n-p=3$, non-trivial examples occur for
$(5,2)$ and $(8,5)$. Moreover, if the $3$-dimensional Poincar\'e
Conjecture is false, then there are non-trivial examples for all
$(n,p)$. If the Poincar\'e Conjecture is true, then
all examples are trivial except $(5,2)$, $(8,5)$ and possibly $(6,3)$.
\end{theorem}

Since the Poincar\'e Conjecture has been
proved to be true, we have that for $n-p = 3$ the map $f$ 
can be non-trivial
only if $(n,p)\in \{(6,3), (8,5), (5,2)\}$.
Therefore, Problem~\ref{problem} has been open
uniquely for the dimension pair $(6, 3)$.

%

In \cite{SHMA} the authors used an extension of 
Milnor-Khimshiashvili's formula proved in \cite{ADD} 
(see Theorem~\ref{euler} of the present paper) 
for real isolated singularity map germs 
to show a manageable characterization of Church-Lamotke's results 
when the Milnor fiber is $3$-dimensional, as follows.

\begin{theorem}\label{teorema 1} 
Let $f:(\mathbb{R}^{n},0) \to (\mathbb{R}^{p},0)$,
$f(x)=(f_{1}(x), f_2(x), \ldots, f_{p}(x))$, be a 
polynomial map germ with an
isolated singularity at the origin, and suppose $n-p=3$. 
Denote by
$\mathrm{deg}_{0}(\nabla f_{1})$ 
the topological degree of the mapping
$$\varepsilon \frac{\nabla f_{1}}{\|\nabla f_{1}\|} : 
S^{n-1}_{\varepsilon} \to
S^{n-1}_{\varepsilon},$$
where
$$\nabla f_1 = \left(\frac{\partial f_1}{\partial x_1},
\frac{\partial f_1}{\partial x_2}, \ldots,
\frac{\partial f_1}{\partial x_{n}}\right).$$

\noindent {\rm(a)} If the pair $(n,p)=(6,3)$, then the following
three are equivalent.
\begin{itemize}
\item[{\rm({i})}] $f$ is trivial.
\item[{\rm({ii})}] $\mathrm{deg}_{0}(\nabla f_{1})=0$.
\item[{\rm({iii})}] The link $K$ is connected.
\end{itemize}

\noindent {\rm(b)} If the pair $(n,p)=(8,5)$, then the
following three are equivalent.
\begin{itemize}
\item[{\rm({i})}] $f$ is trivial.
\item[{\rm({ii})}] $\mathrm{deg}_{0} (\nabla f_{1})=0$.
\item[{\rm({iii})}] The link $K$ is not empty.
\end{itemize}

\noindent {\rm(c)} If the pair $(n,p)=(5,2)$, then 
the following two are equivalent.
\begin{itemize}
\item[{\rm({i})}] $f$ is trivial.
\item[{\rm({ii})}] $\pi_{1}(F_{f})=1$, i.e.\ 
the Milnor fiber $F_{f}$ is simply connected.
\end{itemize}
\end{theorem}


In this paper we aim to give a characterization of NS-pairs
$(S^5, K)$ with $\dim K=2$,
and use it to prove the existence of 
non-trivial real polynomial map germs
$(\R^6, 0) \to (\R^3, 0)$ with an isolated
singularity at the origin, putting an end to
Problem~\ref{problem} posed by Milnor.
For this, we will use tools from configuration spaces 
and a construction by Funar in \cite[Section~2.7]{Funar}. 
More precisely, we first classify fiber bundles
$E^5 \to S^2$ with fiber the $3$-sphere with the interiors
of a disjoint union of
$3$-disks removed, such that the boundary
fibrations are trivial. We will show that the
isomorphism classes of such bundles are in
one-to-one correspondence with the second
homotopy group of a certain configuration space,
and that its elements correspond to a skew-symmetric
integer matrix. Then, we show that a given
fiber bundle $E^5 \to S^2$ is associated with
an NS-pair $(S^5, K)$ if and only if the skew-symmetric
matrix is unimodular.
As a consequence, we see that the number of boundary
components of a fiber is always odd.
Furthermore, this allows us to construct
a lot of non-trivial NS-pairs $(S^5, K)$, and then
the Looijenga construction \cite{Lo} leads to non-trivial
polynomial map germs with an isolated singularity.

Our second aim in this paper is to
introduce necessary and sufficient conditions 
under which the Milnor fiber in the pairs of dimensions $(2n, n)$
and $(2n+1,n)$, $n \geq 3$, is, up to homotopy, a bouquet (or a wedge)
of spheres. 
As applications, we give examples of polynomial map 
germs $(\mathbb{R}^{n},0) \to (\mathbb{R}^{p},0)$, 
$n/2 \geq p \geq 2$, such that the associated Milnor fiber 
is a bouquet of a non-zero number of spheres.

Throughout the paper, the (co)homology groups are with 
integer coefficients unless otherwise specified.
The symbol 
``$\cong$'' denotes a diffeomorphism between
smooth manifolds or an appropriate isomorphism
between algebraic objects.

\section{Classification of bundles}\label{section2}

Let $(S^5, K^2)$ be an NS-pair, where $K^2$ is a closed
$2$-dimensional manifold embedded in the $5$-dimensional
sphere $S^5$. We have the associated
fibration $\pi : S^5 \setminus \Int N(K^2) \to S^2$, where
$N(K^2)$ denotes a closed tubular neighborhood of $K^2$
in $S^5$, and we denote by $F$ its fiber, which
is a compact $3$-dimensional manifold bounded by
a copy of $K^2$.
Since $S^5$ does not fiber over $S^2$, we have $K^2
\neq \emptyset$. 
Furthermore, we have the homotopy exact sequence
$$\pi_2(S^5 \setminus \Int N(K^2)) \to \pi_2(S^2)
\to \pi_1(F) \to \pi_1(S^5 \setminus \Int N(K^2)).$$
Since $\pi$ is trivial on the boundary, it has a section,
so that the homomorphism $\pi_2(S^5 \setminus \Int N(K^2)) \to \pi_2(S^2)$
is surjective. Furthermore, $S^5 \setminus \Int N(K^2)$ is
simply connected. Therefore, the compact
$3$-dimensional manifold $F$ is also simply connected.

Then, by a standard argument, we see that $K^2 \cong \partial F$
consists of some copies of $S^2$ and that $F$ is homotopy equivalent
to a $3$-dimensional sphere with some points removed. 
Then, by the solution to the
Poincar\'e Conjecture, we see that $F$ is diffeomorphic to
$S^3_{(k+1)}$ for some non-negative integer $k$,
where $S^3_{(k+1)}$ denotes the $3$-sphere with
the interiors of $k+1$ disjoint $3$-balls removed.
Therefore, $\pi$ is a smooth fiber bundle with
fiber $S^3_{(k+1)}$ such that it is trivial on the
boundary. In this section, we classify such fiber
bundles.

Let $\mathrm{Diff}(S^3)$ be the topological group of diffeomorphisms
of $S^3$. By the solution to the Smale Conjecture by Hatcher
\cite{Hatcher}, we have that $\mathrm{Diff}(S^3)$ is homotopy
equivalent to the orthogonal group $O(4)$.

Let us denote by $B^3$ the $3$-dimensional closed ball and for a 
non-negative integer $k$, we denote by $\cup^{k+1} B^3$ 
the disjoint union of $k+1$ copies of $B^3$. We sometimes 
regard $\cup^{k+1} B^3$ to
be ``standardly'' embedded in $S^3$, and we denote by 
$j_{k+1}: \cup^{k+1} B^3 \to S^3$ the inclusion map.

We denote by $\mathrm{Emb}(\cup^{k+1} B^3, S^3)$ the space of all smooth
embeddings of $\cup^{k+1} B^3$ into $S^3$, not necessarily the
standard one, and by $\mathrm{Diff}(S^3, \cup^{k+1} B^3)$ the 
subspace of $\mathrm{Diff}(S^3)$ consisting of those 
diffeomorphisms which restrict to the inclusion map 
$j_{k+1}$ on $\cup^{k+1} B^3$.
Furthermore, we denote 
by $\mathrm{Diff}(S^3_{(k+1)}, \partial S^3_{(k+1)})$ the
topological group of diffeomorphisms of $S^3_{(k+1)}$ which
restrict to the identity on the boundary. Note that
$S^3_{(k+1)} = S^3 \setminus \cup^{k+1} \Int{B^3}$.

The lemma below follows from \cite[Proposition~1, p.~120]{Cerf}.

\begin{lemma}\label{lemma1}
The canonical map $\mathrm{Diff}(S^3, \cup^{k+1} B^3) \to
\mathrm{Diff}(S^3_{(k+1)}, \partial S^3_{(k+1)})$ induces
isomorphisms
$$\pi_i(\mathrm{Diff}(S^3, \cup^{k+1} B^3)) \to
\pi_i(\mathrm{Diff}(S^3_{(k+1)}, \partial S^3_{(k+1)}))$$ 
for all $i$.
\end{lemma}

Now consider the natural map
$$\varphi : \mathrm{Diff}(S^3) \to \mathrm{Emb}(\cup^{k+1} B^3, S^3)$$
that sends each diffeomorphism of $S^3$ to its restriction to
$\cup^{k+1} B^3$. The following is a consequence of the Cerf--Palais 
fibration theorem (see \cite[Appendice]{Cerf}, \cite{Palais}).

\begin{lemma}\label{lemma:CP}
The natural map $\varphi$ as above 
is the projection of a locally trivial fiber bundle with fiber
$\mathrm{Diff}(S^3, \cup^{k+1} B^3)$.
\end{lemma}

Therefore, we have the homotopy exact sequence: 
\begin{eqnarray}
& & \pi_2(\mathrm{Diff}(S^3), \id) \to
\pi_2(\mathrm{Emb}(\cup^{k+1} B^3, S^3), j_{k+1}) \to
\pi_1(\mathrm{Diff}(S^3, \cup^{k+1} B^3), \id)
\nonumber \\
& \to & \pi_1(\mathrm{Diff}(S^3), \id) \to
\pi_1(\mathrm{Emb}(\cup^{k+1} B^3, S^3), j_{k+1}) \to \cdots.
\label{eqn1}
\end{eqnarray}

Let $\mathbb{F}_{k+1}(S^3)$ be the
configuration space of $k+1$ points in $S^3$. This space can be 
naturally identified with $\mathrm{Emb}(\{0, 1, \ldots, k\}, S^3)$.

\begin{lemma}
The space $\mathrm{Emb}(\cup^{k+1} B^3, S^3)$ is homotopy
equivalent to $\mathbb{F}_{k+1}(S^3) \times O(3)^{k+1}$.
\end{lemma}

\begin{proof}
For a given embedding $\eta : \cup^{k+1} B^3 \to S^3$, we associate
the element of $\mathbb{F}_{k+1} (S^3)$ which sends the $i$-th 
point to the $\eta$-image of the center of the $i$-th $3$-ball. 
Furthermore, by
associating the normalized differential of $\eta$ at each center, we
get an element of $O(3)^{k+1}$. (Note that the tangent bundle
$TS^3$ of $S^3$ is trivial, and we fix its trivialization here.) 
Then, we can show that the map $\mathrm{Emb}(\cup^{k+1} B^3, S^3) \to
\mathbb{F}_{k+1}(S^3) \times O(3)^{k+1}$ thus obtained gives a homotopy
equivalence. (For example, see 
\cite[Appendice, \S5, Proposition~3]{Cerf}.)
\end{proof}

Recall that $\mathrm{Diff}(S^3) \simeq O(4)$. Furthermore, $\mathbb{F}_{k+1}(S^3)$ is $1$-connected 
(see \cite{FH}, \cite[Proof of Proposition~2.30]{Funar}) 
and $\pi_2(O(3)) = 0$. Thus, the
exact sequence (\ref{eqn1}) turns into
\begin{eqnarray*}
& 0 & \to \pi_2(\mathbb{F}_{k+1} (S^3)) \to 
\pi_1(\mathrm{Diff}(S^3, \cup^{k+1} B^3), \id) 
\to \pi_1(O(4), \id) 
\\
& & \to \pi_1(O(3)^{k+1}, \id) \to \cdots. 
\end{eqnarray*}

Note that $\pi_1(O(4), \id) \cong \mathbb{Z}_2$ and
$\pi_1(O(3)^{k+1}, \id) \cong (\mathbb{Z}_2)^{k+1}$. By choosing the
standard embedding $j_{k+1}$ so that it is equivariant with respect
to the natural $SO(2)$-actions, we see that the homomorphism
$\pi_1(O(4), \id) \to \pi_1(O(3)^{k+1}, \id)$ sends the generator $1 \in
\mathbb{Z}_2$ to $(1, 1, \ldots, 1) \in (\mathbb{Z}_2)^{k+1}$. In 
particular, it is injective. Thus, we have that the boundary
homomorphism
\begin{equation}
\pi_2(\mathbb{F}_{k+1} (S^3)) \to 
\pi_1(\mathrm{Diff}(S^3, \cup^{k+1}  B^3), \id) 
\label{eqn3}
\end{equation}
is an isomorphism.

By \cite{FH} and \cite[Proof of Proposition~2.30]{Funar},
we have the following important result.

\begin{lemma}\label{lemma:iso}
The homotopy group $\pi_2(\mathbb{F}_{k+1}(S^3))$ is isomorphic to
$\mathbb{Z}^{k(k-1)/2}$.
\end{lemma}


Note that, for 
a smooth fiber bundle $S^3_{(k+1)} \hookrightarrow E^5 \to S^2$
with structure group $\mathrm{Diff}(S^3_{(k+1)}, \partial S^3_{(k+1)})$,
its characteristic map 
is an element of 
$$\pi_1(\mathrm{Diff}(S^3_{(k+1)}, \partial S^3_{(k+1)}), \id),$$ 
which is isomorphic to 
$\pi_1(\mathrm{Diff}(S^3, \cup^{k+1} B^3), \id) \cong
\mathbb{Z}^{k(k-1)/2}$
by Lemma~\ref{lemma1}, (\ref{eqn3}) and Lemma~\ref{lemma:iso}.

In fact, given such a smooth fiber bundle
$\pi : E^5 \to S^2$, 
one can consider $S^2=D_1^2 \cup D_2^2$,
where $D_i^2$, $i=1,2$, denote the $2$-dimensional 
closed disk.
Since each $D_i^2$ is contractible, the restriction
$\pi : \pi^{-1} (D_i^2) \to D_i^2$
is a trivial fiber bundle with fiber $S^3_{(k+1)}$, 
and we have $\pi^{-1} (D_i^2) \cong S^3_{(k+1)} \times D_i^2$.
Hence, we can recover the total space 
$$E^5 = (S^3_{(k+1)} \times D_1^2) \cup_h 
(S^3_{(k+1)} \times D_2^2)$$ 
for some diffeomorphism $h: S^3_{(k+1)} \times \partial D_2^2 \to
S^3_{(k+1)} \times \partial D_1^2$ defined by 
$h(x, t)= (\alpha(t)(x), t)$, where $\alpha : S^1 = \partial D^2 
\to \mathrm{Diff}(S^3_{(k+1)}, \partial S^3_{(k+1)})$ 
corresponds to the characteristic map. Therefore, the structure of 
the fiber bundle is 
completely determined by the homotopy class
$[\alpha] \in \pi_1(\mathrm{Diff}(S^3_{(k+1)}, 
\partial S^3_{(k+1)}), \id)$.

\section{Characterization of NS-pairs}\label{section3}

For a non-negative integer $k$, let 
\begin{equation}
S^3_{(k+1)} \hookrightarrow E^5 \stackrel{\pi}{\rightarrow} S^2
\label{eq:bundle}
\end{equation}
be a smooth fiber bundle
such that its restriction to the boundary 
$\partial S^3_{(k+1)} \hookrightarrow \partial E^5 \rightarrow S^2$ 
is a trivial bundle. 
In this section, we characterize such fiber
bundles that arise from an NS-pair $(S^5, K)$ with
$K \cong \cup^{k+1}S^2$.

We start by gluing the trivial bundle 
$\cup^k (B^3 \times S^2) \to S^2$ along the $k$ boundary
components to get the $B^3$-fibration
\begin{equation}
\widetilde{\pi} : Y = E^5 \cup (\cup^k (B^3 \times S^2)) \rightarrow S^2.
\label{eq:tilde-psi}
\end{equation}
This fibration is trivial, since the structure
group $\mathrm{Diff}(B^3, \partial B^3)$ 
is contractible by Hatcher's solution to the Smale Conjecture
\cite{Hatcher}, where $\mathrm{Diff}(B^3, \partial B^3)$ denotes
the space of those diffeomorphisms of $B^3$ which fix
$\partial B^3$ pointwise.
Therefore, the total space $Y$ of the fibration (\ref{eq:tilde-psi})
is diffeomorphic to $B^3 \times S^2$.
Then, by gluing $B^3_0 \times S^2$ to $Y = B^3 \times S^2$ by the map
$\partial B^3_0 \times S^2 \to \partial B^3 \times S^2$ given by
$(x, y) \mapsto (y, x)$, we get the sphere $S^5$,
where $B^3_0$ is a copy of the closed $3$-dimensional ball.
Set $S^2_0 = x_0 \times S^2$, where $x_0$ is the
center of $B^3_0$, and we write $N(S^2_0) = B^3_0 \times S^2$,
which is identified with the closed tubular neighborhood
of $S^2_0$ in $S^5$.

To fix the notation we write $\cup^{k+1} B^3 = \cup^k_{i=0} B^3_i$, 
and denote by $x_i$ the
center of $B^3_i$, where we consider
$\cup^k B^3 = \cup_{i=1}^k B^3_i$.
We also write $S^2_i = x_i \times S^2$, $i = 1, 2, \ldots, k$.
Note that the $2$-spheres $S^2_i$, $i = 0, 1, \ldots, k$, are all
embedded in $S^5$ in a standard way. Furthermore, 
each of $S^2_i$, $i = 1, 2, \ldots, k$, has
linking number $\pm 1$ with $S^2_0$.
In the following discussions, we orient
$S^2_i$, $i = 0, 1, 2, \ldots, k$, in such a way that
the linking number of $S^2_i$ with $S^2_0$
is equal to $+1$, $i = 1, 2, \ldots, k$.

For $y \in S^2$,
we have $\cup^k_{i=1} (B^3_i \times y) \subset 
\widetilde{\pi}^{-1}(y) \cong B^3$. 
Therefore, to each $y \in S^2$ we can naturally associate an element of
the $k$-point configuration space
$\mathbb{F}_k (\Int B^3) \cong \mathbb{F}_k (\mathbb{R}^3)$. 
This defines a \emph{classifying map} 
$c : S^2 \to \mathbb{F}_k(\mathbb{R}^3)$.


Then, we have the following.

\begin{lemma}\label{lemma:cl}
The isomorphism classes of the 
fibrations as in \textup{(\ref{eq:bundle})}
are in one-to-one correspondence with
$\pi_2(\mathbb{F}_k(\mathbb{R}^3)) \cong \Z^{k(k-1)/2}$.
The correspondence is given by associating
the homotopy class of the classifying map $c$. 
\end{lemma}

Recall that according to \cite[Lemma~2.31]{Funar},
a fiber bundle (\ref{eq:bundle}) corresponds to
the element $(\mathrm{lk}(S^2_i, S^2_j))_{1 \leq i < j \leq k}
\in \Z^{k(k-1)/2}$ in the above correspondence, 
where $\mathrm{lk}$ denotes the linking number in $S^5$, and we fix
orientations of $S^2_i$, $i = 1, 2, \ldots, k$, and $S^5$.

\begin{proof}[Proof of Lemma~\textup{\ref{lemma:cl}}]
As has been seen in Section~\ref{section2}, the isomorphism
classes of the bundles in question are in one-to-one
correspondence with $\pi_2(\mathbb{F}_{k+1} (S^3)) \cong
\pi_1(\mathrm{Diff}(S^3, \cup^{k+1}  B^3), \id)$.
On the other hand, it is known that
$\pi_2(\mathbb{F}_k(\mathbb{R}^3))$ is naturally
isomorphic to $\pi_2(\mathbb{F}_{k+1} (S^3))$
(see \cite[p.~38]{FH}). 

Recall the locally trivial fiber bundle
$$\mathrm{Diff}(S^3, \cup^{k+1} B^3) \stackrel{\iota}{\longrightarrow} 
\mathrm{Diff}(S^3) \stackrel{\varphi}{\longrightarrow} 
\mathrm{Emb}(\cup^{k+1} B^3, S^3)$$
of Lemma~\ref{lemma:CP},
where $\iota$ is the natural inclusion map.
For the homotopy class $[\alpha] \in \pi_1(
\mathrm{Diff}(S^3_{(k+1)}, \partial S^3_{(k+1)}))
\cong \pi_1(\mathrm{Diff}(S^3, \cup^{k+1} B^3))$
of the characteristic map,
its $\iota_*$-image vanishes in
$\pi_1(\mathrm{Diff}(S^3))$, so that there exists
a continuous map $\tilde{\alpha} : D^2 \to \mathrm{Diff}(S^3)$
which extends $\iota \circ \alpha : S^1 \to  \mathrm{Diff}(S^3)$. 
Then, the homotopy class of $\varphi \circ \tilde{\alpha} : D^2 \to 
\mathrm{Emb}(\cup^{k+1} B^3, S^3)$ in
$\pi_2(\mathrm{Emb}(\cup^{k+1} B^3, S^3))
\cong \pi_2(\mathbb{F}_{k+1}(S^3)) \cong
\pi_2(\mathbb{F}_k(\R^3))$ is the class
corresponding to $[\alpha]$ by the
isomorphism (\ref{eqn3}). By construction,
this coincides with the homotopy class of the
classifying map $c$.
This completes the proof.
\end{proof}

Now, we have the following natural question.
%
%
%
%
%

\begin{problem}
Which elements of $\pi_2 (\mathbb{F}_k(\mathbb{R}^3))
\cong \mathbb{Z}^{k(k-1)/2}$ correspond to an NS-pair?
\end{problem}

%
We answer this question in our main result in this section, as follows.

\begin{theorem} \label{main-theorem} 
The fiber bundle $S^3_{(k+1)}\hookrightarrow E^5 \stackrel{\pi}{\to} S^2$ as
in \textup{(\ref{eq:bundle})} arises from 
an NS-pair if and only if
$\mathrm{det} \left(\mathrm{lk}(S_i^2,S_j^2)\right)_{1 \leq i,j \leq k} 
= \pm 1$, where $\mathrm{lk}(S_i^2,S_i^2) = 0$ for all $1 \leq i \leq k$
by convention.
\end{theorem}

Note that $\left(\mathrm{lk}(S_i^2,S_j^2)\right)_{1 \leq i,j \leq k}$
is a $k \times k$ skew-symmetric integer matrix.

The rest of this section is devoted to the proof
of the above theorem.

For a fiber bundle (\ref{eq:bundle}), let $F$
be its fiber. We fix the trivialization of the boundary
fibration, and we write
$\partial E^5 = \cup^k_{i=0} (K_i \times S^2)$, where
$K_i \cong S^2$ are the boundary components of $F \cong
S^3_{(k+1)}$ and are
oriented in such a way that the cycle represented by $K_0$ 
is homologous to the sum of the cycles represented by $K_i$, 
$i=1, 2, \ldots, k$. Let $X^5 = 
E^5 \cup (\cup^k_{i=0} (K_i \times B^3))$
be the closed $5$-dimensional manifold obtained by
gluing $E^5$ and $\cup^k_{i=0} (K_i \times B^3)$ along their
boundaries
in such a way that the natural projection
$$\cup^k_{i=0} (K_i \times (B^3 \setminus \{0\}))
\to S^2$$
extends to a smooth fibration
$X^5 \setminus K \to S^2$, where
$K = \cup^k_{i=0} (K_i \times \{0\})$.
Note that in this notation, $K_i$ is identified with
$\partial B^3_i$, $i = 1, 2, \ldots, k$,
and $K_0$ is identified with $\ast \times S^2
\subset \partial B^3_0 \times S^2$.
We warn the reader that the way that $K_i \times B^3$
are attached to $E^5$ is very different from
that for the construction of $Y \subset S^5$ in (\ref{eq:tilde-psi}).

The theorem is a consequence of Lemmas~\ref{lem: 0} and \ref{lem: 1}
below.

\begin{lemma} \label{lem: 0}
The fiber bundle $\pi$ \textup{(\ref{eq:bundle})} 
arises from an NS-pair if and only if
$X^5$ is homotopy equivalent to $S^5$.
\end{lemma}

The above lemma is a consequence of the well-known
fact that every homotopy $5$-sphere is standard \cite{Smale}.

Since we see easily that $X^5$ is simply connected, 
it suffices to study the homology group $H_2(X^5)$.
Now consider the following piece of the Mayer-Vietoris 
exact sequence:
$$H_2 (\cup^k_{i=0} (K_i \times \partial B^3)) 
\stackrel{\rho}{\to} H_2(E^5) \oplus H_2(\cup^k_{i=0} (K_i \times B^3)) 
\to H_2(X^5)\to 0,$$
where the homomorphism $\rho = (i_{1*}, -i_{2*})$ is induced by 
the inclusions $i_{1}: \cup^k_{i=0} (K_i \times \partial B^3) \to E^5$ 
and $i_{2}: \cup^k_{i=0} (K_i \times \partial B^3) \to 
\cup^k_{i=0} (K_i \times B^3)$.

Figure~\ref{fig: S0completa} helps us to understand the images 
of the elements of $H_2(\cup^k_{i=0} (K_i \times \partial B^3))$ 
by the homomorphism $\rho$. Note that this depicts the
situation in $S^5$ and not in $X^5$.



\begin{figure}[t]
\psfrag{S0}{$S_0^2$}
\psfrag{S1}{$S_1^2$}
\psfrag{S2}{$S_2^2$}
\psfrag{Sk}{$S_k^2$}
\psfrag{K0}{$K_0$}
\psfrag{K1}{$K_1$}
\psfrag{K2}{$K_2$}
\psfrag{Kk}{$K_k$}
\psfrag{m0}{$\mu_0$}
\psfrag{m1}{$\mu_1$}
\psfrag{m2}{$\mu_2$}
\psfrag{mk}{$\mu_k$}
\includegraphics[width=0.95\linewidth,height=0.4\textheight,
keepaspectratio]{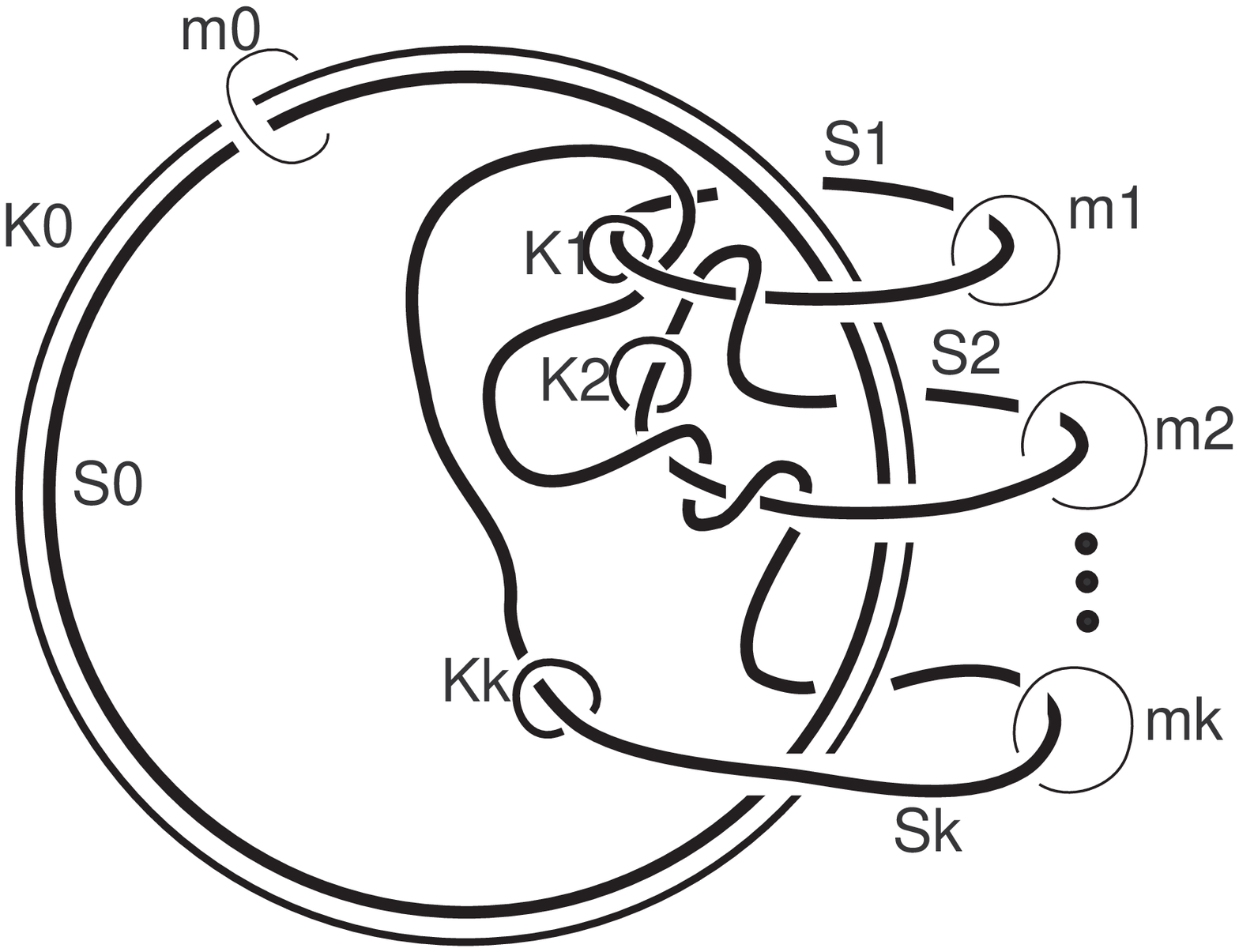}
\caption{Situation in $S^5$}
\label{fig: S0completa}
\end{figure}

In order to describe the homomorphism $\rho$,
let us fix bases of the homology groups. 
In the following, for a cycle $z$, we denote
by $[z]$ the homology class represented by $z$.
First, we have
$$H_2(\cup^k_{i=0} (K_i \times \partial B^3)) \cong 
\oplus^k_{i=0} H_2(K_i \times \partial B^3)$$
and each $H_2(K_i \times \partial B^3) \cong \Z \oplus \Z$
is generated by $[K_i \times \ast]$ and
$[y_i \times \partial B^3]$, where $y_i \in K_i$, 
$i = 0, 1, \ldots, k$. Furthermore, we have
$$H_2(\cup^k_{i=0} (K_i \times B^3)) \cong 
\oplus^k_{i=0} H_2(K_i \times B^3)$$
and each $H_2(K_i \times B^3) \cong \Z$
is generated by $\delta_i  =
[K_i \times \ast]$, $\ast \in B^3$, $i = 0, 1, \ldots, k$.
On the other hand, we have
$$E^5 = S^5 \setminus (\cup^k_{i=0} \Int N(S^2_i)),$$
where $N(S^2_i) = B^3_i \times S^2$ is the closed tubular
neighborhood of $S^2_i$ in $S^5$,
and $S^2_i = x_i \times S^2$, $i = 1, 2, \ldots, k$, are 
so-called ``Hopf duals'' to $S^2_0$.
Therefore, by Alexander duality we have
$$H_2(E^5) \cong H^2 (\cup^k_{i=0} \Int N(S^2_i)).$$
Since $N(S^2_i) = B^3_i \times S^2$, we can take the
generators $\mu_i = [\partial B^3_i \times \ast]
\in H_2(E^5)$, $\ast \in S^2$, and $H_2(E^5)$
is freely generated by $\mu_i$, $i = 0, 1, \ldots, k$.
Here, we orient $\mu_i$ in such a way that the
linking number of $\mu_i$ with $S^2_i$ is equal to $+1$.
Observe that $\mu_i = i_{1*}([K_i \times \ast])$ for 
$i=1, 2, \ldots, k$, and 
$\delta_i = i_{2*}([K_i \times \ast])$ 
for $i=0, 1, \ldots, k$. Therefore, the images of the generators 
by the homomorphism $\rho$ can be written as follows.
\begin{eqnarray*}
\rho([K_0 \times \ast]) & = & \mu_1 + \mu_2 + \cdots + 
\mu_k - \delta_0, \\
\rho([K_i \times \ast]) & = & \mu_i - \delta_i \quad (1\leq i \leq k), \\
\rho([y_0 \times \partial B^3]) & = & \mu_0 + 0, \\
\rho([y_i \times \partial B^3]) & = & 
\sum_{0 \leq j \leq k, j \neq i} \mathrm{lk}(S^2_i, S^2_j)\mu_j + 0 
\quad (1 \leq i \leq k).
\end{eqnarray*}

Therefore, with respect to the above bases,
the homomorphism $\rho$ is represented by the following matrix:
\begin{equation*}
R = \left(
\begin{array}{ccccccccl}
0 & 0 & \cdots & 0 &  & 1 & 1 & \cdots & 1 \\
1 & 1 & \cdots & 0 & & 0 & a_{11} & \cdots & a_{1k} \\
\vdots & \vdots & \ddots & \vdots &  & \vdots &
\vdots & \ddots & \vdots \\
1 & 0 & \cdots & 1 &  & 0 & a_{k1} & \cdots & a_{kk}\\
 &  &  &  &  & &  &  & \\
-1 & 0 & \cdots & 0 &  & 0 & 0 & \cdots & 0\\
0 & -1 & \cdots & 0 &  & 0 & 0 & \cdots & 0\\
\vdots & \vdots & \ddots & \vdots &  & \vdots & \vdots & \ddots & \vdots\\
0 & 0 & \cdots & -1 &  & 0 & 0 & \cdots & 0 \\
\end{array}
\right), 
\end{equation*}
where $a_{ij} = \mathrm{lk}(S^2_j , S^2_i)$, $i \neq j$, 
$1 \leq i,j\leq k$ and $a_{ii}=0$. Observe that $a_{ij}=-a_{ji}$.

\begin{lemma} \label{lem: 1} 
The $5$-dimensional manifold $X^5$ is homotopy
equivalent to $S^5$ if and only if $\mathrm{det}\,R =\pm 1$.
\end{lemma}

\begin{proof}
If $X^5$ is homotopy equivalent to $S^5$, then
its second homology group must vanish and
therefore the homomorphism $\rho$ must
be an epimorphism, which implies that $\mathrm{det}\,R =\pm 1$.

Conversely, if $\rho$ is an isomorphism, by the
above Mayer-Vietoris exact sequence, we have
$H_2(X^5) = 0$. Then by Poincar\'e duality,
we see that $X^5$ has the homology of $S^5$.
Then, a standard argument in
algebraic topology shows that $X^5$ 
is homotopy equivalent to $S^5$.
%

This completes the proof of Lemma~\ref{lem: 1},
and hence Theorem~\ref{main-theorem} has been proved. 
\end{proof}

Since a skew-symmetric integer matrix has determinant
$\pm 1$ only if its size is even, we have
the following.

\begin{corollary}
If the fiber bundle $S^3_{(k+1)}\hookrightarrow E^5 \to S^2$ 
as in \textup{(\ref{eq:bundle})} arises from an NS-pair, 
then $k$ must be even.
\end{corollary}

\begin{remark} \label{lobs}
In \cite{Lo} Looijenga showed how to use the connected sum of NS-pairs 
to construct new ones. In fact, he proved that given 
an NS-pair $(S^n, K^{n-p-1})$ with fiber $F$, 
there exists a polynomial map germ $f: (\mathbb{R}^{n+1},0) 
\to (\mathbb{R}^{p+1},0)$ 
with an isolated singularity at the origin such that
the associated NS-pair is isomorphic to the connected sum 
$$(S^n, K^{n-p-1}) \sharp ((-1)^{n-1}S^n, (-1)^{n-p}K^{n-p-1})$$ 
with fiber being diffeomorphic to the interior of 
$\overline{F} \natural (-1)^{n-p}\overline{F}$, where 
``$\natural$'' means the connected sum along the boundary. 
For further details the reader is referred to \cite[p.~421]{Lo}.
\end{remark}

The following proposition follows from the remark above and the previous result.

\begin{proposition}\label{non-trivial}
For every even integer $k \geq 0$, there exists
an NS-pair $(S^5, L_{k+1})$ with $L_{k+1}$ being
diffeomorphic to the disjoint union of $k+1$ copies of $S^2$, 
and there exists a polynomial map germ $f:(\mathbb{R}^6, 0) \to (\mathbb{R}^3,0)$
with an isolated critical point at $0$ such that the 
associated NS-pair is isomorphic to
$(S^5, L_{k+1} \sharp (-L_{k+1}))$. 
In particular, $L_{k+1} \sharp (-L_{k+1})$ consists of $2k+1$
connected components.
\end{proposition}

\begin{proof} 
First note that for each positive even integer
$k$, there exists a skew-symmetric integer matrix
of determinant $\pm 1$. (For example,
consider the direct sum of the matrix
$$
\begin{pmatrix} 
0 & 1 \\ -1 & 0
\end{pmatrix}
$$
and its copies.)
Then, by the above argument,
there exists an NS-pair $(S^5, L_{k+1})$ corresponding
to that matrix. Now, one can just apply Looijenga's 
construction explained above in Remark~\ref{lobs}. 
\end{proof}

\begin{corollary}
Given a real polynomial map germ as in
Proposition~\textup{\ref{non-trivial}} with $k > 0$, 
the fiber of the associated Milnor fibration is not 
diffeomorphic to a disk.
\end{corollary}

This answers to Milnor's non-triviality question,
Problem~\ref{problem},
for the dimension pair $(6, 3)$.

\section{A generalization to higher dimensions}\label{section4}

We can generalize the construction of Section~\ref{section3}
in higher dimensions as follows, in order to obtain
new non-trivial examples of real polynomial map
germs with an isolated singularity.

Let $n \geq 3$ be an integer.
For a non-negative integer $k$, let
$S^n_{(k+1)}$ denote the $n$-dimensional
sphere $S^n$ with the interior of the disjoint
union of $k+1$ copies of the $n$-dimensional disks removed.
In this section, we will construct a smooth fiber bundle
$$S^n_{(k+1)} \hookrightarrow E^{2n-1} \stackrel{\pi}{\to} S^{n-1}$$
such that the restriction to the boundary
$$\partial S^n_{(k+1)} \hookrightarrow \partial E^{2n-1} 
\stackrel{\pi}{\to} S^{n-1}$$
is a trivial bundle and that it arises from
an NS-pair $(S^{2n-1}, K^{n-1})$.

Let $A = (a_{ij})$ be a $k \times k$ integer matrix
which is $(-1)^n$-symmetric such that the
diagonal entries all vanish.
Let $S_0 \cong S^{n-1}$ be a trivially embedded oriented
$(n-1)$-sphere in $S^{2n-1}$. Then, there exist
mutually disjoint
smoothly embedded oriented $(n-1)$-spheres
$S_i$ in $S^{2n-1}$, $i = 1, 2, \ldots, k$, 
such that
\begin{itemize}
\item[$(1)$] $S_i$ do not intersect $S_0$,
\item[$(2)$] $S_i$ have linking number $+1$ with $S_0$,
\item[$(3)$] the linking number $\mathrm{lk}(S_i, S_j)
= a_{ij}$, $i \neq j$.
\end{itemize}
Such embeddings do exist (for example,
see \cite{H3}).
Note that then
$$E^{2n-1} = S^{2n-1} \setminus \cup_{i=0}^k
\Int N(S_i)$$
naturally fibers over $S^{n-1}$ in such a way that
the restriction to the boundary is trivial.
(More precisely, consider 
the associated sub-fibration of the
trivial fiber bundle $S^{2n-1} \setminus \Int N(S_0)
\cong B^n \times S^{n-1} \to S^{n-1}$.)
Then, by the same construction as in Section~\ref{section3},
we obtain an object $(X^{2n-1}, K^{n-1})$, where
$X^{2n-1}$ is a $(2n-1)$-dimensional smooth closed manifold,
$X^{2n-1} \setminus \Int{N(K^{n-1})}$ is diffeomorphic to $E^{2n-1}$, and it
fibers over $S^{n-1}$ with fiber $S^n_{(k+1)}$ in such a way
that the projection map is compatible with a trivialization of 
the closed tubular neighborhood $N(K^{n-1})$.
Then, we have the following.

\begin{lemma}
The manifold $X^{2n-1}$ is a homotopy $(2n-1)$-sphere if and only if
$\mathrm{det}\,A = \pm 1$.
\end{lemma}

\begin{proof}
We see easily that $E^{2n-1}$, and hence $X^{2n-1}$
is $(n-2)$-connected. Thus, $X^{2n-1}$ is a homotopy
$(2n-1)$-sphere if and only if $H_{n-1}(X^{2n-1})$ vanishes.
Then, an argument using a Mayer-Vietoris exact sequence
as in the previous section leads to the desired result.
\end{proof}

Combining this with the Looijenga construction (Remark~\ref{lobs}),
we have the following.

\begin{corollary}\label{cor:hdim}
Let $n \geq 3$ be an integer.
For every positive integer $k$ with $k \equiv 1 \pmod 4$, 
there exists an NS-pair $(S^{2n-1}, L_k)$ with $L_k$ being
diffeomorphic to the disjoint union of $k$ copies of $S^{n-1}$, 
and there exists a polynomial map germ $f:(\mathbb{R}^{2n}, 0) \to (\mathbb{R}^n,0)$
with an isolated singularity at $0$ such that the 
associated NS-pair is isomorphic to
$(S^{2n-1}, L_k \sharp (-1)^{n}L_k)$. 
In particular, $L_k \sharp (-1)^{n}L_k$ consists of $2k-1$
connected components.
\end{corollary}

Note that the associated Milnor fiber is
diffeomorphic to $S^n_{(2k-1)}$ and is homotopy
equivalent to the bouquet of $2k-2$ copies of the
$(n-1)$-sphere.

\begin{proof}[Proof of Corollary~\textup{\ref{cor:hdim}}]
Set $\ell = (k-1)/2$, which is a nonnegative even integer.
There exists an $\ell \times \ell$ $(-1)^n$-symmetric integer matrix $A$
with determinant $\pm 1$. By Haefliger \cite{H3},
there exists an embedding of $\cup_{i=0}^k S^{n-1}_i$ into
$S^{2n-1}$ such that each component is embedded trivially,
that $S^{n-1}_i$ links with $S_0$ once for all $i > 0$,
and that the linking matrix for $\cup_{i=1}^\ell S^{n-1}_i$
coincides with $A$, where each $S^{n-1}_i$ is a copy
of $S^{n-1}$ and the linking number $\mathrm{lk}(S^{n-1}_i,
S^{n-1}_i) = 0$, $i = 1, 2, \ldots, \ell$, by
convention. Then by the above construction,
we get the homotopy sphere $X^{2n-1}$ in which the disjoint
union of $\ell + 1$ copies of the $(n-1)$-spheres is embedded.
Then, the connected sum $X^{2n-1} \sharp (-X^{2n-1})$ is
diffeomorphic to $S^{2n-1}$ by \cite{KM}, since $n \geq 3$.
Therefore, by the connected sum construction,
we get an NS-pair $(S^{2n-1}, L_k)$, where $L_k$
is diffeomorphic to the disjoint union of $2\ell+1 = k$
copies of $S^{n-1}$.

Then, applying Looijenga's construction,
we get the desired conclusion.
\end{proof}



\section{Bouquet theorems for real isolated singularities}

It is known that for a holomorphic function germ 
$f: (\mathbb{C}^{n+1}, 0) \to (\mathbb{C}, 0)$
with an isolated singularity at the origin, the Milnor fiber 
$F_{f}$ has the homotopy type of a bouquet (or a wedge) of $n$-dimensional spheres. 
For real polynomial map germs with an isolated singularity, 
we cannot expect, in general, such a bouquet theorem, which
can be seen as follows.

By Zeeman's twist spinning construction \cite{Z},
one can construct an NS-pair $(S^4, K^2)$ such that
the fundamental group of the fiber
is not a free group. Then Looijenga's construction
leads to a non-trivial polynomial map germ
$(\R^5, 0) \to (\R^2, 0)$ with an isolated singularity
at the origin such that the Milnor fiber does not
have a free fundamental group. Consequently, the Milnor
fiber is not homotopy equivalent to a bouquet of spheres.
%

\begin{remark}\label{spinning}
In the following,
in order to get examples in higher dimensions,
we use the spinning construction due to Artin
\cite{Artin}. For completeness, let us recall the
construction.
Let $(S^m, K^k)$ be an NS-pair
with $K^k \neq \emptyset$ and $\pi : S^m 
\setminus K^k \to S^{m-k-1}$ the associated
fibration. We denote the fiber of the fibration
$S^m \setminus \Int N(K^k) \to S^{m-k-1}$ by $F^{k+1}$,
where $N(K^k)$ is the closed tubular neighborhood of $K^k$
in $S^m$.
We take a point $q \in K^k$ and a small
$m$-disk neighborhood $D$ in $S^m$ such that
$(D, D \cap K^k)$ is diffeomorphic to the
standard disk pair $(D^m, D^k)$ and that
$\pi$ restricted to $D \setminus (D \cap K^k)$
is equivalent to the standard fibration $D^m \setminus
D^k \to S^{m-k-1}$.
Then, we consider the quotient space
of $(S^m \setminus \Int D, K^k \setminus (\Int D \cap K^k))
\times S^1$, where for each $x \in \partial D$,
the points of the form $(x, t)$ are identified
to a point for all $t$. This kind of a construction
is called the \emph{spinning}.
The resulting pair gives $(S^{m+1}, \tilde{K}^{k+1})$,
where $\tilde{K}^{k+1}$ is a smoothly embedded submanifold
of $S^{m+1}$ of dimension $k+1$. By construction,
there exists a fibration $\tilde{\pi} : S^{m+1}
\setminus \tilde{K}^{k+1} \to S^{m-k-1}$
which restricts to $\pi$ on $(S^m \setminus (\Int D \cup K^k)) 
\times \{t\}$
for each $t \in S^1$. It is straightforward to see
that $(S^{m+1}, \tilde{K}^{k+1})$ is an NS-pair.
We call it the \emph{spun}
of the NS-pair $(S^m, K^k)$.
Note that the fiber $\tilde{F}^{k+2}$
of the fibration $S^{m+1} \setminus
\Int N(\tilde{K}^{k+1}) \to S^{m-k-1}$
is diffeomorphic to the $(k+2)$-dimensional
manifold obtained from $F^{k+1} \times S^1$
by identifying, for each $x \in \Delta^k$, the points
of the form $(x, t)$ to a point for all $t$, where
$\Delta^k$ is a $k$-dimensional
disk embedded in $\partial F^{k+1}$ (near $q$).
Note that the fundamental groups of $S^m \setminus K^k$
and $S^{m+1} \setminus \tilde{K}^{k+1}$ are isomorphic,
and that $F^{k+1}$ and $\tilde{F}^{k+2}$ also
have isomorphic fundamental groups. 
\end{remark}


Let $(S^4, K^2)$ be an NS-pair such that the fiber
has non-free fundamental group.
Then, applying once the spinning construction explained 
above to $(S^4, K^2)$, one gets a non-trivial example in dimension 
$(6,2)$ such that the Milnor fiber is not homotopy equivalent to
a bouquet of spheres. 
Performing such procedures inductively one can construct 
examples in all pairs of dimensions $(n, 2)$, $n \geq 5$,
such that the Milnor fiber is not homotopy
equivalent to a bouquet of spheres. 

In this section we give sufficient conditions to 
guarantee that the real Milnor fiber is 
homotopy equivalent to a bouquet of spheres of 
the same dimension, or of different dimensions.

Throughout this section we consider $f: (\mathbb{R}^{n}, 0) 
\to (\mathbb{R}^{p}, 0)$, $n > p \geq 2$, 
a polynomial map germ with an isolated singularity 
at the origin and the Milnor fibration (the ``Milnor tube"),
$$f: f^{-1} (S_{\delta}^{p-1}) \cap D_{\varepsilon}^{n}
\to S_{\delta}^{p-1},$$
where $0 < \delta \ll \varepsilon \ll 1$. 
We denote by $F_{f}$ its fiber and by 
$\beta_{j}=\rk H_{j}(F_{f})$ its $j$-th Betti number.

Consider $\pi :(\mathbb{R}^{p},0) \to (\mathbb{R}^{p-1},0)$, 
$p \geq 3$, the germ of the canonical projection. Clearly,
the composition map germ $G=\pi \circ f :(\mathbb{R}^{n},0) 
\to (\mathbb{R}^{p-1},0)$ also has an isolated singularity 
at the origin and thus we have two fibrations:
$$f: f^{-1} (S_{\delta}^{p-1}) \cap D_{\varepsilon}^{n}
\to S_{\delta}^{p-1},$$ 
and
$$G: G^{-1} (S_{\delta}^{p-2}) \cap D_{\varepsilon}^{n}
\to S_{\delta}^{p-2}.$$

In \cite{DA} it was shown the relationship between the 
Milnor fibers $F_{f}$ and $F_{G}$.
%
It is worth pointing out that the results in 
\cite{DA} hold in a more general setting which includes 
the case of non-isolated singularities. 
Nevertheless, in the special case of an isolated singularity, 
it provides a positive answer to a conjecture stated by 
Milnor in \cite[p.~100]{Mi} as follows:

\begin{theorem}[\textup{\cite{DA}}] \label{proj}
Let $f:(\mathbb{R}^{n},0)\to (\mathbb{R}^{p},0)$, $n > p \geq 2$, 
be a polynomial map germ with an isolated singularity at the
origin and set $G=\pi \circ f :(\mathbb{R}^{n},0) \to 
(\mathbb{R}^{p-1},0)$. Then, the Milnor fiber $F_{G}$ of $G$
is homeomorphic to $F_{f}\times [0,1]$, where for $p=2$ the Milnor fiber 
of $G$ is, by definition, the intersection of a sufficiently
small closed ball centered at the origin and the
inverse image of a regular value 
sufficiently close to the origin. In particular, 
the Milnor fibers $F_f$ and $F_G$ have the same homotopy type.
\end{theorem}

In \cite[Chapter~11]{Mi}, Milnor provided information concerning 
the topology of the fiber $F_{f}$. It was proved in 
Lemma~11.4 that if $n < 2(p-1)$, then the Milnor fiber 
is necessarily contractible. It also follows from the 
first paragraph of the proof that for $n > p \geq 2$ in general, 
if the link is not empty, then the fiber $F_{f}$ is 
$(p-2)$-connected, i.e., $\pi_{i}(F_{f})=0$, $i=0,1, \ldots, p-2$.

In \cite{ADD} the authors proved formulae relating 
the Euler characteristic of the Milnor fiber and the 
topological degree of the gradient mapping of the coordinate 
functions, which extends Milnor's formula for
complex function germs with an isolated singularity 
(see \cite[p.~64]{Mi}) and Khimshiashvili's formula \cite{Khim} 
for isolated singularity real analytic function germs, as follows.

\begin{theorem}[\cite{ADD}]\label{euler}
Let $f:(\mathbb{R}^{n},0) \to (\mathbb{R}^{p},0)$, $n > p \geq 2$,
be a polynomial map germ with an isolated singularity at the origin, 
and consider 
$$f(x)=(f_{1}(x), f_2(x), \ldots, f_{p}(x)),$$
an arbitrary representative of the germ. Denote by 
$\mathrm{deg}_{0}(\nabla f_{i}(x))$, for $i=1, 2, \ldots, p$, 
the topological degree of the map 
$\displaystyle{\varepsilon\frac{\nabla f_{i}}{\|
\nabla f_{i} \|}:S^{n-1}_{\varepsilon} \to S^{n-1}_{\varepsilon}}$, 
for $\varepsilon > 0$ small enough.
\begin{itemize}
\item[(i)] If $n$ is even, then 
$\chi (F_{f}) = 1-\mathrm{deg}_0 \nabla f_1$. 
Moreover, we have
$$\mathrm{deg}_0 \nabla f_{1}= \mathrm{deg}_0 \nabla f_{2}
= \cdots= \mathrm{deg}_0 \nabla f_{p}.$$
\item[(ii)] If $n$ is odd, then 
$\chi (F_{f}) = 1$. Moreover, we have
$\mathrm{deg}_0 \nabla f_{i}=0$ for $i=1, 2, \ldots, p$.
\end{itemize}
\end{theorem}

In particular, from item (ii) above it follows that 
if the source space is odd-dimensional, then the 
fiber can never be homotopy equivalent to
a bouquet of a positive number
of spheres of the same dimension.

In the following subsections, we consider the dimension
pairs $(2n, n)$ and $(2n+1, n)$, and study
conditions for a Milnor fiber to have the homotopy type
of a bouquet of spheres. We also
study the dimension pairs $(2n, p)$ and $(2n+1, p)$
with $2 \leq p \leq n$ using the composition with a
projection.

\subsection{The case of $(\mathbb{R}^{2n}, 0) \to (\mathbb{R}^n, 0)$}
\label{even}

Consider $f: (\mathbb{R}^{2n}, 0) \to (\mathbb{R}^n, 0)$,
$n\geq 2$, a polynomial map germ with an isolated
singularity at the origin. Note that $S^{2n-1}$
does not smoothly fiber over $S^{n-1}$. Hence, in this case
$F_{f}$ is an $n$-dimensional compact orientable 
manifold with non-empty boundary and $\pi_{i}(F_{f})=0$ 
for $i=0,1, \ldots, n-2$.
Since $\partial F_{f} \neq \emptyset$, we have
$H_n(F_{f}) = 0$.
Moreover, since $F_{f}$ is
orientable, the homology $H_{n-1}(F_{f})$ is torsion free. 
Then, by Theorem~\ref{euler}, item (i),
we have
$\beta_{n-1}=(-1)^{n}\mathrm{deg}_{0}(\nabla f_{1})$.

Furthermore, in the special case $n=2$, the fibers are compact
connected surfaces with non-empty boundary, so that
they have the homotopy type of a bouquet of $1$-dimensional 
spheres (circles). Furthermore, for $n=3$, we have
seen in Section~\ref{section2}
that the fibers are diffeomorphic to
$S^3_{(k+1)}$ for some non-negative integer $k$,
and hence they are homotopy equivalent to
a bouquet of $2$-spheres.
Therefore, we
may assume that $n \geq 4$.
Note that if $\mathrm{deg}_{0}(\nabla f_{1}) = 0$,
then the Milnor fiber is contractible.

It follows from the Hurewicz theorem that the 
Hurewicz homomorphism 
$$\rho_{n-1} : \pi_{n-1} (F_{f}) \to  H_{n-1} (F_{f}) 
\cong \mathbb{Z}^{\beta_{n-1}}$$
is an isomorphism. Then, for each generator 
$\gamma_i \in H_{n-1}(F_{f}) \cong \mathbb{Z}^{\beta_{n-1}}$ 
there exists a continuous map $\varphi_i : S^{n-1} \to F_f$,
$i = 1, 2, \ldots, \beta_{n-1}$, such that 
$\gamma_i = \rho_{n-1} ([\varphi_i])= (\varphi_i)_* ([S^{n-1}])$, 
where $[S^{n-1}] \in H_{n-1} (S^{n-1})\cong \mathbb{Z}$ is 
the fundamental class (given by
the natural orientation of $S^{n-1}$). 
Therefore, we have the continuous map
$$\varphi:\displaystyle{\bigvee^{\beta_{n-1}} S^{n-1} \to F_{f}}$$
obtained by the wedge of the maps $\varphi_i: S^{n-1} \to F_{f}$, for
$i= 1, 2, \ldots, \beta_{n-1}$, which is a homotopy 
equivalence by the Whitehead theorem.


Thus we have proved the following:

\begin{proposition} \label{wedge-even} 
Let $f : (\R^{2n}, 0) \to (\R^n, 0)$ be a polynomial
map germ with an isolated singularity at the origin, $n \geq 2$.
Given $f(x)=(f_{1}(x), f_2(x), \ldots, f_{n}(x))$, a 
representative of the germ $f$, we have the following.
\begin{itemize}
\item[(i)] $\beta_{n-1} = (-1)^n \mathrm{deg}_{0}(\nabla f_{1})$.
\item[(ii)] The Milnor fiber $F_{f}$ has the homotopy type of 
a bouquet of $(n-1)$-dimensional spheres 
$$\displaystyle {\bigvee^{\beta_{n-1}} S^{n-1}},$$ 
where it means a point when $\beta_{n-1}=0$.
\end{itemize}
\end{proposition}

%
For $n \geq 4$, it follows from Theorem~\ref{T2}, item (c), 
that in all pairs of dimensions $(2n,n)$ there exist non-trivial 
examples. 
However, these non-trivial examples due to
Church--Lamotke \cite{CL} have contractible Milnor fibers
(with non-simply connected links). 
On the other hand, according to our
construction in Section~\ref{section4} together with
Theorem~\ref{proj},
we get the following.

\begin{corollary}\label{bouquet}
For each pair of dimensions $(2n, p)$, $2 \leq p \leq n$, 
there exists a real isolated singularity polynomial map germ 
$(\mathbb{R}^{2n}, 0)\to (\mathbb{R}^p, 0)$ such that the 
Milnor fiber is, up to homotopy, a bouquet of $(n-1)$-dimensional
spheres with the number of
spheres equal to $|\mathrm{deg}_{0}(\nabla f_{1})| > 0$,
where $f(x)=(f_{1}(x), f_2(x), \ldots, f_p(x))$.
\end{corollary}

\subsection{The case of $(\mathbb{R}^{2n+1}, 0) \to (\mathbb{R}^n, 0)$}
\label{odd} 

Consider now $f: (\mathbb{R}^{2n+1}, 0) \to (\mathbb{R}^n, 0)$, 
$n\geq 3$, a polynomial map germ with an isolated
singularity at the origin. In this case, the Milnor fiber $F_{f}$ 
is an $(n+1)$-dimensional compact orientable manifold with 
non-empty boundary and is
$(n-2)$-connected. Then, $H_{n+1}(F_{f})=0$, 
$H_{n}(F_{f})$ is torsion free and, by Theorem~\ref{euler}, 
$\beta_{n}=\beta_{n-1}$.
Suppose that $H_{n-1}(F_{f})$ is torsion free. Then, we have $H_{n-1}
(F_f)\cong \mathbb{Z}^{\beta_{n-1}} \cong H_n (F_f)$. By Hurewicz 
theorem, the Hurewicz homomorphisms
$$\rho_{n-1} : \pi_{n-1} (F_{f}) \to H_{n-1} (F_{f}) \cong \mathbb{Z}^{\beta_{n-1}}$$
and
$$\rho_n : \pi_n (F_{f}) \to H_n (F_{f}) \cong \mathbb{Z}^{\beta_{n-1}}$$
are surjective. Then, by an argument similar to that
used in the case $(2n, n)$, we can construct a homotopy equivalence
$$\varphi : \left(\displaystyle{\bigvee^{\beta_{n-1}}S^{n-1}}\right) \vee 
\left(\displaystyle{\bigvee^{\beta_{n}}S^{n}}\right)
\to F_f.$$
Thus, we have proved the following result.

\begin{proposition}\label{wedge-odd}
Let $f:(\mathbb{R}^{2n+1},0)\to (\mathbb{R}^{n},0)$, $n\geq 3$, 
be a real isolated singularity polynomial map germ.
Then, the $(n-1)$-th homology $H_{n-1}(F_{f})$ of the
Milnor fiber is torsion free if and only if $F_f$
has the homotopy type of 
a bouquet of spheres of the form
$$\left(\displaystyle{\bigvee^{\beta_{n-1}}S^{n-1}}\right)
\vee
\left(\displaystyle{\bigvee^{\beta_{n-1}}S^{n}}\right)
= \bigvee^{\beta_{n-1}} (S^{n-1} \vee S^n),$$
where it means a point when $\beta_{n-1}=0$.
%
\end{proposition}

According to our
construction in Section~\ref{section4} together with
Theorem~\ref{proj} again, we get the following.

\begin{corollary}\label{bouquet2}
For each pair of dimensions $(2n+1, p)$, $2 \leq p \leq n$, 
there exists a real isolated singularity polynomial map germ 
$(\mathbb{R}^{2n+1}, 0)\to (\mathbb{R}^p, 0)$ such that the 
Milnor fiber is, up to homotopy, a bouquet of $\ell$ copies
of the $n$-dimensional
sphere and $\ell$ copies of the $(n-1)$-dimensional sphere
with $\ell > 0$.
\end{corollary}

\begin{proof}
For $n \geq 3$, this is a consequence of
Proposition~\ref{wedge-odd}.
For $n = 2$, we start with a non-trivial
fibered knot $(S^3, K)$. Then, its spun $(S^4, \tilde{K})$
is a non-trivial fibered $2$-knot,
and its fiber is obtained by spinning a positive
genus surface with boundary. Therefore, the fiber
of $(S^4, \tilde{K})$ has the homotopy
type of a bouquet of a positive number of circles
and $2$-spheres. This completes the proof.
\end{proof}


\subsection{Application to $k$-stairs maps}
Given a polynomial map germ $f:(\mathbb{R}^n, 0)\to (\mathbb{R}^q, 0)$,
$n \geq q \geq 1$,
with an isolated singularity at the origin, we say that a map 
germ $F: (\mathbb{R}^{n}, 0) \to (\mathbb{R}^{p}, 0)$,
$1 \leq q \leq p$, is a 
\emph{$(p-q)$-stairs map for $f$} if there exist germs of 
polynomial functions $g_{j}: (\mathbb{R}^{n}, 0) \to (\mathbb{R}, 0)$, 
$q+1 \leq j  \leq p$, such that $F(x)=(f(x), g_{q+1}(x), g_{q+2}(x), 
\ldots, g_{p}(x))$ has an isolated singularity at the origin. 
If $p=q$, then by definition, we have $F(x)=f(x)$ and $f$ 
is its own $0$-stairs map.

\begin{corollary}
Let $f:(\mathbb{R}^{n},0)\to (\mathbb{R}^{p},0)$, 
$n/2 \geq p \geq 2$, 
be a polynomial map germ with an isolated singularity 
at the origin. Then we have the following.
\begin{itemize}
\item [(i)] If $n$ is even and $f$ admits a 
$(n/2-p)$-stairs map, then the Milnor fiber 
is homotopy equivalent to a bouquet of $(n/2-1)$-dimensional spheres.
\item [(ii)] Suppose $n$ is odd and $H_{k}(F_{f})$ is torsion free 
for $k=(n-1)/2-1$, where $F_f$ denotes the Milnor fiber. 
If $f$ admits a $((n-1)/2-p)$-stairs map, 
then the Milnor fiber is homotopy equivalent to a bouquet of 
$k$- and $(k+1)$-dimensional spheres, where the numbers of
spheres are the same.
\end{itemize} 
\end{corollary}

\begin{proof} 
Just apply Propositions~\ref{wedge-even}, \ref{wedge-odd}, and
Theorem~\ref{proj}.
\end{proof}

We do not know whether or not the bouquet structure in 
the fiber characterizes the existence of such $k$-stairs maps
for $k \geq 1$.

\section*{Acknowledgements}

The authors would like to express their sincere gratitude to
Edivaldo Lopes dos Santos and Denise de Mattos
for stimulating discussions and invaluable comments.

The first author would like to thank very much the Fapesp grant 2013/23443-5 
and CNPq/PQ-2 309819/2012-1. 
The second author would like to thank very much all Fapesp support given during 
the development of the PhD project, grant number 2012/12972-4. 
The third author has been supported in part by JSPS KAKENHI Grant Number
23244008, 23654028. 
The fourth author would like to thank CNPq/PDJ grant 502638/2012-5 during 
the post-doc in ICMC/USP.


\end{document}